\documentclass[11pt]{amsart}

\usepackage{epigamath}

\usepackage[english]{babel}

\numberwithin{equation}{section}

\usepackage[shortlabels]{enumitem}
\setlist[enumerate,1]{label={\rm(\alph*)}, ref={\rm\alph*}}

\usepackage{danpackE}

\usetikzlibrary{cd}
\usepackage[utf8]{inputenc}
\DeclareUnicodeCharacter{0301}{}

\usepackage{cleveref}
\usepackage[normalem]{ulem}
\newtheorem{theorem}{Theorem}[section]
\newtheorem{proposition}[theorem]{Proposition}
\newtheorem{corollary}[theorem]{Corollary}
\newtheorem{lemma}[theorem]{Lemma}

\newtheorem*{theoremA*}{Theorem A}
\newtheorem*{theoremB*}{Theorem B}

\theoremstyle{remark}
\newtheorem{question}[theorem]{Question}
\newtheorem{remark}[theorem]{Remark}
\newtheorem*{remark*}{Remark}

\renewcommand{\phi}{\varphi}
\renewcommand{\epsilon}{\varepsilon}
\DeclareMathOperator{\Spec}{Spec}
\renewcommand{\hat}{\widehat}
\DeclareMathOperator{\eval}{ev}
\DeclareMathOperator{\ass}{Ass}

\newcommand{\fp}{\mathfrak p}
\newcommand{\fa}{\mathfrak a}

\renewcommand{\[}{\begin{equation*}}
\renewcommand{\]}{\end{equation*}}

\newcommand{\surj}{\twoheadrightarrow}
\newcommand{\lra}{\longrightarrow}

\renewcommand{\compSubAlg}{\, \scalebox{1}{\rotatebox[origin=c]{90}{$\circlearrowleft$} } }

\newcommand{\algCompWithIdeal}[3]{\sC_{#3}(#1, #2 \compSubAlg)}
\newcommand{\algCompWithIdealFull}[2]{\sC(#1, #2 \compSubAlg)}
\newcommand{\restrictedCartAlgSpecifyR}[3]{ {\algCompWithIdeal{#1}{#2}{#3}_{/ {#2}}} }
\newcommand{\restrictedCartAlgSpecifyRFull}[2]{ {\algCompWithIdealFull{#1}{#2}}_{/{#2}} }

\newcommand{\algCompWithI}[1]{\algCompWithIdeal{R}{I}{#1}}
\newcommand{\algCompWithIFull}{\algCompWithIdealFull{R}{I}}
\newcommand{\restrictedCartAlg}[1]{\restrictedCartAlgSpecifyR{R}{I}{#1}}
\newcommand{\restrictedCartAlgFull}{ \restrictedCartAlgSpecifyRFull{R}{I} }

\DeclareMathOperator{\Gr}{Gr}
\DeclareMathOperator{\GL}{GL}
\newcommand{\supth}[1]{\ensuremath{#1^{\mathrm{th}}}}

\EpigaVolumeYear{8}{2024} \EpigaArticleNr{1} \ReceivedOn{August 15, 2022}
\InFinalFormOn{September 4, 2023}
\AcceptedOn{October 11, 2023}

\title{Diagonal F-splitting and symbolic powers of ideals}
\titlemark{Diagonal F-splitting and symbolic powers of ideals}

\author{Daniel Smolkin}
\address{University of Michigan,  
  Department of Mathematics,
530 Church St, 
Ann Arbor, MI 48109-1043, USA}
\email{smolkind@umich.edu}

\authormark{D.~Smolkin}

\AbstractInEnglish{Let $J$ be any ideal in a strongly $F$-regular, diagonally $F$-split ring $R$ essentially of finite type over an $F$-finite field. We show that $J^{s+t} \subseteq \tau(J^{s - \epsilon}) \tau(J^{t-\epsilon})$ for all $s, t, \epsilon > 0$ for which the formula makes sense. We use this to show a number of novel containments between symbolic and ordinary powers of prime ideals in this setting, which includes all determinantal rings and a large class of toric rings in positive characteristic. In particular, we show that $\fp^{(2hn)} \subseteq \fp^n$ for all prime ideals $\fp$ of height $h$ in such rings.}



\begin{document}

%%%%%%%%%%%%%%%%%%%%%%%%%%%%%%%
% Title page
%%%%%%%%%%%%%%%%%%%%%%%%%%%%%%%

%\removeabove{}
%\removebetween{}
%\removebelow{}

\maketitle

\begin{prelims}

\DisplayAbstractInEnglish

\bigskip

\DisplayKeyWords

\medskip

\DisplayMSCclass

%\bigskip

%\languagesection{Fran\c{c}ais}

%\bigskip

%\DisplayTitleInFrench

%\medskip

%\DisplayAbstractInFrench

\end{prelims}

%%%%%%%%%%%%%%%%%%%%%
% Table of Contents
%%%%%%%%%%%%%%%%%%%%%

\newpage

\setcounter{tocdepth}{1}

\tableofcontents

%%%%%%%%%%%%%%%%%%%%%
% Content begins here
%%%%%%%%%%%%%%%%%%%%%

\section{Introduction}
Given an ideal $I$ in a ring $R$, we define the $\supth{n}$ \emph{symbolic power} of $I$ to be 
\[
  I^{(n)} = \bigcap_{\fp \in \ass_R(R/I)} I^n R_{\fp} \cap R. 
\]
If $I$ is radical, then one can think of $I^{(n)}$ as roughly being the ideal of elements in  $R$ vanishing to ``order at least $n$'' along the variety $V(I)$. For instance, if $I$ is a divisorial ideal $I = R(D)$, then $I^{(n)} = R(nD)$; see also \cite[Proposition 2.14]{SymbolicPowersSurvey}.  One quickly sees from the definitions that $I^n$ is always contained in $I^{(n)}$, but finding a more precise relationship between ordinary and symbolic powers of ideals is a subtle matter; see \cite{SymbolicPowersSurvey, GrifoSeceleanuSymbolicRees} for excellent overviews of this problem and its history.

Here we address the following question posed in \cite{HunekeKatzValidashtiUniformEquivIsolatedSing}.  

\begin{question}
  Let $(R, \mathfrak m)$ be  complete local domain. Does there exist an integer $C \geq 1$ such that $\fp^{(Cn)} \subseteq \fp^n$ for all $\fp \in \Spec R$ and all integers $n \geq 1$?
\end{question}

So far we know this is true, under mild assumptions,  for isolated singularities (see \cite{HunekeKatzValidashtiUniformEquivIsolatedSing}), Segre products of polynomial rings (see \cite{USTPDiagFReg}), and certain Hibi rings (see \cite{HibiDFR}). We also know that this property descends under finite extensions of rings (see \cite{HunekeKatzValidshtiUSTPFiniteExtensions}). See also Walker's work on monomial ideals and height 1 primes in toric rings \cite{WalkerUniformSymbolicNormalToric}. 

Rings with an affirmative answer to this question are said to have the \emph{uniform symbolic topology property}, or USTP for short. This is because one can use the symbolic powers of an ideal $\fp$ to define a topology on $R$, which is generated by cosets $x + \fp^{(n)}$ for $x \in R$ and $n \geq 1$. Compare this to the $\fp$-adic topology, which is generated by the cosets $x + \fp^n$. Schenzel showed that these two topologies are equivalent, meaning 
\[
  \forall a \in \bN\colon\quad \exists b \in \bN\colon \fp^{(b)} \subseteq \fp^a
\]
whenever $R$ is a Noetherian domain; see \cite{SchenzelSymbolicPowersTopology}. Swanson later showed that, in this case, one can find a number $C(\fp)$, depending on $\fp$, where $\fp^{(C(\fp)n)} \subseteq \fp^n$ for all $n$; see \cite{SwansonLinearEquivalence}. In this sense, the two topologies are said to be \emph{linearly equivalent}. Ein, Lazarsfeld, and Smith kicked off the study of USTP when they showed that one can find a \emph{uniform} $C$ such that $\fp^{(Cn)} \subseteq \fp^n$ for all $\fp \in \Spec R$ and all $n \geq 1$ whenever $R$ is a regular finitely generated $\bC$-algebra; see \cite{ELSsymbolicPowers}. Analogous results were later shown in positive characteristic by Hochster--Huneke \cite{HHsymbolic} and in mixed characteristic by Ma--Schwede \cite{MaSchwedePerfectoid}.

Test ideals of pairs, and their characteristic zero/mixed characteristic analogs, have been an important tool in the study of USTP; see \cite{ELSsymbolicPowers, HaraYoshidaSubadd, TakagiSingMultIdeals, USTPDiagFReg,  MaSchwedePerfectoid, MurayamaSymbolic}. Each of those proofs relies on some version of the ``subadditivity formula,'' 
  \[
    \tau(\mathfrak a^s \mathfrak b^t) \subseteq \tau(\mathfrak a^s) \tau( \mathfrak b^t).
  \] 
  In regular rings, this formula holds as written, but in other situations one must account for its failure. In \cite{TakagiSingMultIdeals}, Takagi multiplies by the Jacobian ideal. In \cite{USTPDiagFReg}, we use so-called ``diagonal Cartier algebras'' and show that a ring has USTP if it is \emph{diagonally $F$-regular}. Diagonal $F$-regularity is a strengthening of \emph{diagonal $F$-splitting} (and of strong $F$-regularity), analogously to the way strong $F$-regularity is a strengthening of $F$-splitting. So far, however, interesting examples of diagonally $F$-regular rings have been quite difficult to find. 

  In this work, we prove that if a ring is only strongly $F$-regular and diagonally $F$-split, then it satisfies a weaker form of subadditivity which is enough for USTP. Namely, we show  the following. 

\begin{theoremA*}
  Let $\kay$ be an $F$-finite field of positive characteristic, and let $R$ be a strongly $F$-regular, diagonally $F$-split $\kay$-algebra essentially of finite type. Let $\mathfrak a\subseteq R$ be an ideal, and let $s$ and $t$ be positive real numbers with $s + t \in\bZ$. Then we have
  \[
    \mathfrak a^{s + t} \subseteq \tau( \mathfrak a^{s - \epsilon} ) \tau( \mathfrak a^{t- \epsilon})
  \]
  for all $\epsilon \in (0, \min\{s, t\} ]$. 
\end{theoremA*}

From this, we follow a line of argument similar to that used in \cite{USTPDiagFReg} to get various bounds on symbolic powers; see~\Cref{thm:strongerSymbolicPowers} for the strongest statement. In particular, we have the following. 

\begin{theoremB*}
  In the same setting as Theorem A, let $\mathfrak p \subseteq R$ be a prime ideal and $h = \height \mathfrak p $. Then $\mathfrak p^{(2hn)} \subseteq \mathfrak p^n$ for all $n > 0$. 
  \label{thm:symbolicPowers}
\end{theoremB*}

Observe that Theorem B implies $\fp^{(2dn)} \subseteq \fp^n$ for all $\fp$ and all $n$, where $d$ is the dimension of $R$. 

By work of Ramanathan \cite{RamanathanDiagFSplit} and Lauritzen--Raben-Pedersen--Thomsen \cite{GloballyFRegSchubert}, the hypotheses of these theorems are satisfied by any affine subset of a Schubert subvariety of $G/Q$, where $G$ is a connected and simply connected linear algebraic group over an algebraically closed field of positive characteristic, and $Q$ is a parabolic subgroup of $G$ containing a Borel subgroup $B$.  In particular, these theorems hold in any determinantal ring $\kay[X_{m\times n}]/I_t$, where $X_{m\times n}$ is an $m\times n$ matrix of indeterminates, with $m \leq n$, and $I_t$ is the ideal of size $t$ minors of $X$, where $t \leq m$. (We show that $\kay$ need not be algebraically closed here in~\Cref{cor:diagFSplitBaseChg}.) These theorems also hold in diagonally $F$-split toric rings, which were characterized in \cite{PayneUnimodular}. This includes, in particular, all Hibi rings (\emph{cf.} \cite[Theorem 3.7]{HibiDFR}).

By standard ``reduction mod $p$'' techniques, the symbolic powers containment of Theorem B holds in the corresponding determinantal  and toric rings over fields of characteristic zero; see \cite[Section~2]{HHequalCharZero} (see also \cite[Section~6]{USTPDiagFReg}).

\subsection*{Acknowledgements} I would like to thank Elo\'isa Grifo, Melvin Hochster, Srikanth Iyengar, Karl Schwede, George Seelinger, and Karen Smith for their helpful comments and conversations. Thanks also to Pedro Ram\'irez-Moreno and Francisco Espinoza for pointing out an error in Lemma 2.2, and to the referees for many helpful comments and corrections. 

\section{Background on test ideals}
We begin with some background on positive-characteristic commutative algebra. For a more in-depth treatment on this material, see  the surveys \cite{SmithZhangSurvey} and \cite{SchwedeTuckerSurvey}, or the recent lecture notes \cite{HochsterFundTcl}. 

Let $R$ be a ring of characteristic $p$, where $p$ is prime. Then the Frobenius map, 
\[
  F\colon R \lra R, \quad x \longmapsto x^p,
\]
is a ring homomorphism. So are its iterates $F^e(x) \coloneqq x^{p^e}$ for all $e \in \bZ_{\geq 0}$. We write $F^e_*$ to denote restriction of scalars along $F^e$. More precisely, $F^e_*$ is a functor from $R$-modules to $R$-modules. For any $R$-module $M$, we have $F^e_* M = \{ F^e_* m \mid m \in M\}$, where addition and scalar multiplication are defined by $F^e_* m + F^e_* n = F^e_*(m+n)$ and $rF^e_*m = F^e_* (r^{p^e}m)$ for all $r\in R$ and $m,n \in M$. For any $R$-linear map $\vp \colon M \to N$, the map $F^e_* \vp \colon F^e_* M \to F^e_* N$ is given by $(F^e_* \vp)(F^e_* m) = F^e_*(\vp(m))$. If $R$ happens to be a domain, then we can fix an algebraic closure $K$ of $\operatorname{frac}(R)$, and each element of $R$ will have a (unique!) $\supth{(p^e)}$ root in $K$. In this case, $F^e_* R$ is nothing but $R^{1/p^e}$, the set of these $\supth{(p^e)}$ roots with the usual $R$-module structure. We say that $R$ is \emph{$F$-finite} if $F^e_*R$ is finitely generated over $R$ for some (equivalently, all) $e > 0$. 

An \emph{$F$-splitting} of $R$ is a map of $R$-modules $F^e_* R \to R$ sending $F^e_* 1$ to $1$. Such a map splits the natural $R$-module map $R \to F^e_* R$ given by $r \mapsto r F^e_* 1 = F^e_* r^{p^e}$---in this sense, it splits the Frobenius map. We say $R$ is \emph{$F$-split} if it admits an $F$-splitting $F^e_* R \to R$ for some (equivalently, all) $e > 0$. Note that this is equivalent to the existence of an $R$-linear surjection $F^e_* R \surj R$. $F$-split rings are automatically reduced. 

\subsection*{Global assumptions} For the rest of this section, $R$ denotes an $F$-finite, reduced, and Noetherian ring of characteristic $p > 0$. 

\subsection{Test ideals and strong \texorpdfstring{$\boldsymbol{F}$}{F}-regularity} 
Let $R^\circ$ be the set of elements not in any minimal prime of $R$. Two important notions from Hochster and Huneke's celebrated theory of tight closure are \emph{big test elements} and \emph{big test ideals}. A big test element is defined to be any $c \in R^\circ$ such that, for all $d \in R^\circ$, there exist $e> 0$ and $\vp\in \homgp_R(F^e_*R, R)$ with $\vp(F^e_* d) = c$. It is far from obvious, but big test elements exist in any $F$-finite, reduced, Noetherian ring, and the ideal generated by all big test elements of $R$ is called the \emph{big test ideal} of $R$, denoted by $\tau(R)$. It is given by
\[
  \tau(R) \coloneqq \sum_{e > 0} \sum_{\vp \in \homgp_R(F^e_* R, R)} \vp\left( F^e_* c \right),
\]
where $c \in R^{\circ}$ is any test element. One of the more ``magic,'' and extremely useful, aspects of $F$-singularity theory is that this ideal is independent of the choice of $c$ (though this is a quick consequence of the definitions; perhaps the real magic is that test elements exist at all). We say $R$ is \emph{strongly $F$-regular} if $\tau(R) = R$, or equivalently if $1$ is a big test element of $R$. Strong $F$-regularity is a local condition, and strongly $F$-regular local rings are normal Cohen--Macaulay domains. If $S$ is a faithfully flat $R$-algebra and $S$ is strongly $F$-regular, then so is $R$; see \cite[Lecture 10]{HochsterFundTcl}. 

The reader should be warned that the big test ideal is sometimes denoted by $\tilde \tau(R)$ in the literature, such as in \cite{HaraTakagiGeneralizationOfTestIdeals}, whereas the notation $\tau(R)$ is reserved for the ``ordinary'' test ideal. These two ideals are conjectured to be the same, however. In this paper we will only work with big test ideals and big test elements. Thus, we will follow earlier authors in omitting the qualifier ``big''  throughout.

\subsection{Test ideals of pairs}
Test ideals are, in a precise sense, the positive-characteristic analog of multiplier ideals from birational geometry. However, birational geometers have long worked with multiplier ideals not just of schemes $X$, but of pairs $(X, Z)$, where $Z$ is a formal $\bQ$-linear combination of subschemes of $X$. This inspired Hara and Yoshida to define the test ideals  of pairs: if  $\fa$ is an ideal of $R$ not contained in any minimal prime, and $t \geq 0$ is a real number, we say that $c\in R^\circ$ is a \emph{test element of the pair} $(R, \fa^t)$ if, for all $d\in R^{\circ}$, there exist some $e > 0$ and some $\vp \in \homgp_R(F^e_*R, R)$ with 
\[
  c \in \vp\left( F^e_* d \fa^{\ceil{t(p^e-1)}} \right).
\]
Then we can define the \emph{test ideal of the pair} $(R, \fa^t)$ to be
\[
  \tau(R, \fa^t) = \sum_{e > 0 } \sum_{\vp\in \homgp_R(F^e_* R, R)} \vp\left( F^e_* c \fa^{\ceil{t(p^e-1)}} \right),
\]
where $c$ is any test element of the pair $(R, \fa^t)$. We will denote the test ideals by $\tau(\fa^t)$ when the ambient ring $R$ is clear from context. 

These definitions can certainly feel quite daunting and unmotivated to the uninitiated. Fortunately, all of the facts we will need about test ideals can be summarized as follows.  

\begin{proposition}[Properties of test ideals] 
  \label{prop:basicFacts}
  Let $\fa$ be an ideal not contained in any minimal prime of $R$, and let $t \geq 0$. Then there exists a test element of the pair $(R, \fa^t)$; see \cite[Lemma 2.1]{HaraTakagiGeneralizationOfTestIdeals}, \emph{cf.} \cite[Proposition 3.21]{SchwedeNonQGor}. Further, we have the following: 
 \begin{enumerate}
    \item If\, $W\subseteq R$ is any multiplicative set, then $W\invrs \tau(R, \fa^t) = \tau(W\invrs R, (W\invrs \fa)^t)$; see \cite[Proposition~3.1]{HaraTakagiGeneralizationOfTestIdeals}. 
    \item For all $n\in \bZ_{\geq 0}$, $\tau( (\fa^n)^t) = \tau( \fa^{nt})$; see \cite[Remark 6.2]{HaraYoshidaSubadd}. 
    \item If $\fa$ has a reduction generated by $r$ elements, then $\tau(\fa^{t+n}) = \fa^n \tau(\fa^t)$ for all $n \in \bZ_{\geq 0}$ and all $t \geq r-1$; see \cite[Theorem 4.2]{HaraTakagiGeneralizationOfTestIdeals}.       \label{p:bF-c}
    \item If $c \in R$ is a big test element and $\vp \in \homgp_R(F^e_* R, R)$, then $\vp(F^e_* c \fa^{\ceil{tp^e}}) \subseteq \tau(\fa^t)$; see \cite[Lemma~2.1]{HaraTakagiGeneralizationOfTestIdeals}, cf. \cite[Proposition 3.3(5)]{BSTZ_discreteness_and_rationality}. \label{part:vpOfTestElement}
  \end{enumerate}
\end{proposition}

\begin{remark*}
  Recall that an ideal $J\subseteq I$ is called a \emph{reduction} of $I$ if $J I^n = I^{n+1}$ for all $n$ sufficiently large. In Noetherian rings, this is equivalent to saying $\overline J = \overline I$, where $\overline I$ denotes the integral closure of $I$. See \cite[Chapter 8]{HunekeSwansonIntegral}
  \end{remark*}

\begin{proof}
  We prove part \eqref{part:vpOfTestElement}, for the reader's convenience. Let $d$ be a test element of the pair $(R, \fa^t)$. Then, since $c$ is a big test element of $R$, there exists a $\psi\colon F^{e_0}_* R \to R$ with $\psi(F^{e_0}_* d) = c$. It follows that
  \[
    \vp(F^e_* c \fa^{\ceil{tp^e}}) = \vp\left( F^e_* \fa^{\ceil{tp^e}} \psi(F^{e_0}_* d) \right) \subseteq \vp\left( F^e_*  \psi(F^{e_0}_* \fa^{\ceil{tp^e}p^{e_0}}d) \right).
  \]
  Note that $\ceil{tp^e}p^{e_0}$ is an integer and that
  \[
    \ceil{tp^e}p^{e_0} \geq t p^e p^{e_0} > t\left( p^{e+e_0}-1 \right).
  \]
  It follows that $\ceil{tp^e}p^{e_0} \geq \ceil{t( p^{e+e_0}-1 )}$. Finally, since the map
  \[
    F^{e+e_0}_* R \lra R, \quad F^{e+e_0}_* x\longmapsto \vp\left( F^e_* \psi(F^{e_0}_*x) \right)
  \]
  is $R$-linear, we have the desired containment.
\end{proof}

Using  Proposition~\ref{prop:basicFacts}, we can prove the following containment, which is key in establishing the connection between test ideals and symbolic powers. 

\begin{lemma}
  Let $\fp \in \Spec R$ be an ideal of height $h$ with infinite residue field $R_\fp/\fp R_\fp$ \textup{(}\,for instance, $\fp$ can be any nonmaximal prime ideal of $R$\textup{)}. Then
  \[
    \tau( (\fp^{(N)})^t ) \subseteq \fp^{(\floor{Nt} - h+1)}
  \]
  for all integers $N > 0$ and all real numbers $t > 0$ such that $Nt \geq h-1$. 
  \label{lemma:tauSymbolicPower}
\end{lemma}

\begin{proof}
 We can check this containment locally at $\fp$. Localizing, we get 
  \[
    \tau(R,  (\fp^{(N)})^t )R_\fp  = \tau(R_\fp, (\fp^{(N)} R_\fp)^t) =\tau(R_\fp,  (\fp^N R_\fp)^t) =  \tau( R_\fp, (\fp R_{\fp})^{Nt} ).
  \]
  As the residue field at $\fp$ is infinite, we know that $\fp R_\fp$ has a reduction generated by $h$ elements (see {\cite[Theorem 8.3.7, Corollary 8.3.9]{HunekeSwansonIntegral}}).  Since test ideals localize, it follows from~\Cref{prop:basicFacts}, part~\eqref{p:bF-c}, that
\[
  \tau( R_\fp, (\fp R_\fp)^{Nt})  \subseteq  \tau( R_\fp, (\fp R_\fp)^{\floor{Nt}})  =  \fp^{\floor{Nt} - h + 1} \tau(R_\fp,  (\fp R_\fp)^{h-1})  \subseteq \fp^{\floor{Nt}-h+1}R_\fp,
\]
as desired. 
\end{proof}

\subsection{Tensor products}
Let $A$ be a ring of characteristic $p$ and $R$ an $A$-algebra. 
For all $R$-modules $M$ and $N$, we have a natural surjection of $R\otimes_A R$-modules $\Theta\colon F^e_*M \otimes_A F^e_* N \surj F^e_*(M \otimes_A N)$, just given by $F^e_* m \otimes_A F^e_*n \mapsto F^e_* (m\otimes_A n)$. The kernel of this map is the submodule
  \[
    \left \langle F^e_* am \otimes_A F^e_* n - F^e_* m \otimes_A F^e_* an \mid a \in A, m\in M, n \in N \right \rangle \subseteq F^e_* M \otimes_A F^e_* N.
  \]
  For instance, if  $A$ and $R$ are domains, then $F^e_* (R \otimes_A R) = R^{1/p^e}\otimes_{A^{1/p^e}} R^{1/p^e}$, and this is a homomorphic image of $F^e_* R\otimes_A F^e_*R  = R^{1/p^e} \otimes_A R^{1/p^e}$. If $A = A^p$ (\textit{i.e.}, the Frobenius map $A \to A$ is surjective), then $\Theta$ is an isomorphism. If $A$ is an $F$-finite field and $R$ is finitely generated over $A$, then maps $F^e_* (R \otimes_A R) \to R\otimes_A R$ can be expressed in terms of maps $F^e_* R \to R$, as follows. 

  \begin{lemma}
  Let $\kay$ be an $F$-finite field and $R$ a $\kay$-algebra essentially of finite type. Then we have a natural inclusion 
  $
    \homgp_{R\otimes_\kay R}(F^e_* (R\otimes_\kay R), R\otimes_\kay R) \subseteq \homgp_R(F^e_* R, R)\otimes_\kay \homgp_R(F^e_* R, R).
  $
  \label{lemma:homOfTensor}
\end{lemma}

  \begin{proof}
  We have a natural surjection $F^e_* R\otimes_\kay F^e_* R \surj F^e_*(R\otimes_\kay R)$, which induces an inclusion
  \[
    \homgp_{R\otimes_\kay R}( F^e_* (R \otimes_\kay R), R\otimes_\kay R) \subseteq \homgp_{R\otimes_\kay R}(F^e_* R \otimes_\kay F^e_*R, R \otimes_\kay R).
  \]
  By \cite[Lemma 3.9]{SmolkinSubaddPublished}, we have a natural isomorphism 
  \[
    \homgp_{R\otimes_\kay R}(F^e_* R \otimes_\kay F^e_*R, R \otimes_\kay R) \cong \homgp_R(F^e_* R, R) \otimes_\kay \homgp_R(F^e_* R, R). \qedhere
  \]
\end{proof}

\section{Cartier algebras, compatible splittings, and diagonal \texorpdfstring{$\boldsymbol{F}$}{F}-splitting}
In this section we define diagonal $F$-splitting and the closely related diagonal Cartier algebra. Using the language of diagonal Cartier algebras allows us to state a base-change lemma for diagonal $F$-splitting (see \Cref{cor:diagFSplitBaseChg}). Without this base-change lemma, we would need to assume that $\kay$ is algebraically closed in~\Cref{thm:detlDiagFSplit}.

\subsection{Cartier algebras}
Cartier algebras give us a unified framework for discussing test ideals in a much more general setting than just pairs $(R, \fa^t)$; see \cite{SchwedeNonQGor}. We review some of the basics of this theory. 

Let $R$ be a ring of characteristic $p$. For each $e$, define $\sC_e(R) = \homgp_R(F^e_* R, R)$. 
Then $\sC(R)= \bigoplus_e \sC_e(R)$ is a graded, noncommutative ring: given $\vp \in \sC_e(R)$ and $\psi \in \sC_d(R)$, we define the multiplication in $\sC(R)$ by
\[
  (\vp \cdot \psi)(F^{e+d}_* x) = \vp(F^e_* \psi(F^d_* x)).
\]
We define a \emph{Cartier algebra} on $R$ to be a graded subring $\mathscr D \subseteq \sC(R)$  with degree zero piece $\mathscr D_0 = \homgp_R(R,R)$. We say that a Cartier algebra $\mathscr D$ is $F$-split if $\mathscr D_e$ contains an $F$-splitting for some $e > 0$. Note that this is equivalent to the existence of a surjective map $\vp \in \mathscr D_e$ for some $e$: if $\vp$ is surjective, then $\vp(F^e_*x) = 1$ for some $x \in R$. Letting $m_x\in \homgp_R(R, R)$ be ``multiplication by $x$,'' we see that  $\vp\cdot m_x \in \mathscr D_e$ and $(\vp \cdot m_x) (F^e_* 1) = 1$. Further, if $\vp \in \mathscr D_e$ is an $F$-splitting, then $\vp^n \in \mathscr D_{ne}$ is an $F$-splitting for all $n > 0$. 

\subsection{Splittings compatible with an ideal}
One sort of Cartier algebra, which is particularly relevant to us, is the Cartier algebra of maps compatible with a given ideal. Given a map $\vp \colon F^e_* R \to R$ and an ideal $I \subseteq R$, we say $\vp$ is \emph{compatible with $I$} if $\vp(F^e_* I) \subseteq I$. If we set
\[
  \algCompWithI{e} = \left\{ \vp \in \sC_e(R) \mid \vp(F^e_* I) \subseteq I \right\},
\]
then one quickly checks that $\algCompWithIFull \coloneqq \bigoplus_e \algCompWithI{e}$ is a Cartier algebra on $R$; see  \cite[Proposition~3.4]{SmolkinSubaddPublished}. We call this the \emph{Cartier algebra on $R$ of maps compatible with $I$}.

One special property of this Cartier algebra is that maps in $\algCompWithI{e}$ restrict to maps in $\algCompWithI{d}$ when $d < e$, in the following sense: for any $d < e$, we have a natural map $\iota_d^e\colon F^d_* R \to F^e_*R$ given by $\iota_d^e(F^d_* x) = F^e_* x^{p^{e-d}}$. Observe that $\iota_d^e(F^d_* I) \subseteq F^e_* I$. Thus, for any $\psi \in \algCompWithI{e}$, we have $\psi \circ \iota_d^e \in \algCompWithI{d}$. Note that if $R$ is a domain, then $\psi \circ\iota_d^e$ is literally the restriction of $\psi\colon R^{1/p^e}\to R$ to $R^{1/p^d}$. %This sort of property was proposed as an object of study in \cite[Remark 3.7]{SchwedeNonQGor}. 

We say that $R$ is \emph{$F$-split compatibly with $I$} if the Cartier algebra $\algCompWithIFull$ is $F$-split. If this is the case, then $\algCompWithI{d}$ contains an $F$-splitting for all $d > 0$. Indeed, for any $d>0$, we can find some $e>d$ and an $F$-splitting $\psi \in \algCompWithI{e}$. Then $\psi \circ \iota_d^{e}$ is an $F$-splitting in $\algCompWithI{d}$. 

The Cartier algebra $\algCompWithIFull$ induces a Cartier algebra on $R/I$: each map $\vp \in \algCompWithI{e}$ induces a map $\overline \vp\colon F^e_*(R/I) = F^e_* R/F^e_* I \to R/I$ by  $\overline \vp(F^e_* (x +I)) = \vp(x) + I$. Following \cite[Definition 3.1]{SmolkinSubaddPublished} (but slightly altering the notation), we define
\[
  \restrictedCartAlg{e} = \{ \overline \vp \mid \vp \in \algCompWithI{e}\}.
\]
Then $\restrictedCartAlgFull \coloneqq \bigoplus_e \restrictedCartAlg{e}$ is a Cartier algebra on $R/I$; see \cite[Proposition 3.2]{SmolkinSubaddPublished}. We call this the restriction\footnote{This terminology comes from the geometric picture, where $\overline \vp$ can be seen as the restriction of $\vp$ to the subscheme $V(I)$.} of $\algCompWithIFull$ to $R/I$. Put another way, given a map $\psi\colon F^e_* (R/I) \to R/I$, we say that a map $\widehat \psi\colon F^e_* R \to R$ is a \emph{lift of\, $\psi$ to $R$} if these maps fit into a commutative diagram
\[
  \xymatrix{
    F^e_* R \ar[d] \ar[r]^{\widehat \psi} & R \ar[d] \\
    F^e_* (R/I) \ar[r]^{\psi} & R/I\rlap{,}
  }
\]
where the downward arrows are the canonical surjections. Then we get
\[
  \restrictedCartAlg{e} = \{ \psi \in \sC_e(R/I) \mid \psi \textrm{ has a lift to }R\}.
\]
The map $\iota_d^e$ is a lift of the natural map $\overline \iota_d^e\colon F^d_*(R/I) \to F^e_*(R/I)$, so $\psi \circ \overline \iota_d^e$ is an element of $\restrictedCartAlg{d}$ whenever $\psi$ is an element of $\restrictedCartAlg{e}$. By the same argument as before, it follows that $\restrictedCartAlg{e}$ contains an $F$-splitting for each $e > 0$ whenever $\restrictedCartAlgFull$ is $F$-split. We further have the following.

\begin{lemma}
  \label{lemma:restrictSplitting}
  Let $R$ be a ring of characteristic $p$ and $I$ an ideal of $R$. Consider the following two statements:
  \begin{enumerate} 
    \item\label{lrS-a} $\algCompWithI{e}$ contains an $F$-splitting of $R$. 
    \item\label{lrS-b} $\restrictedCartAlg{e}$ contains an $F$-splitting of $R/I$.
  \end{enumerate}
  Then \eqref{lrS-a} implies \eqref{lrS-b}. If\, $R$ is $F$-split, then \eqref{lrS-b} implies \eqref{lrS-a}. 
\end{lemma}

\begin{proof}
  If $\vp \in \algCompWithI{e}$ is an $F$-splitting, then $\overline \vp$ is an $F$-splitting in $\restrictedCartAlg{e}$. For the other direction, suppose that $\psi \in \restrictedCartAlg{e}$ is an $F$-splitting of $R/I$ and suppose that $\alpha$ is an $F$-splitting of $R$. Then $\psi$ has some lift $\widehat \psi \in \algCompWithI{e}$ with $\widehat \psi(F^e_* 1) \in 1 + I$. Let $i = \widehat \psi(F^e_* 1) - 1$. Then $\widehat \psi - i \alpha$ is an $F$-splitting and $\widehat \psi - i \alpha \in \algCompWithI{e}$. 
\end{proof}

\subsection{Diagonal \texorpdfstring{$\boldsymbol{F}$}{F}-splitting}
Let $A$ be an $F$-finite ring, and let $R$ be an $A$-algebra essentially of finite type. Let $\mu_A$ be the multiplication map, 
\[
  \mu_A \colon R \otimes_A R \lra R, \quad \mu_A(x \otimes_A y) = xy.
\]
Then $R$ is defined to be \emph{diagonally $F$-split over $A$} if $R\otimes_A R$ is $F$-split compatibly with $\ker \mu_A$. We define the \emph{second diagonal Cartier algebra on $R$ over $A$} to be the Cartier algebra
\[
  \sD 2(R/A) \coloneqq \restrictedCartAlgSpecifyRFull{R\otimes_A R}{\ker \mu_A};
\]
see \cite[Notation 3.7]{SmolkinSubaddPublished}, \cite[Definition 3.1]{USTPDiagFReg}. The upshot is that if $R$ is diagonally $F$-split over $A$, then the Cartier algebra $\sD 2(R/A)$ is $F$-split, which implies that for all $e > 0$, we can find $F$-splittings $\vp$ and $\overline \vp$ making the following diagram commute:
\begin{equation*}
\xymatrix{
F_*^e (R \otimes_A R) \ar[r]^-{\vp} \ar[d]_-{F^e_* \mu_A} & R\otimes_A R \ar[d]^-{\mu_A}\\ 
 F_*^e R\ar[r]^-{\overline \vp} & R\rlap{.}
} \label{eq:splittingdiagram}
\end{equation*}

\begin{remark}
  If $A = A^p$ and $\vp$ is any $F$-splitting on $R$, then $\vp\otimes_A \vp$ is a well-defined $F$-splitting on $R \otimes_A R$ because $F^e_*(R\otimes_A R) = F^e_* R \otimes_A F^e_* R$. Then, by~\Cref{lemma:restrictSplitting}, $R$ is diagonally $F$-split over $A$ if and only if $\sD 2(R/A)$ is $F$-split. This amends an error in \cite[Remark 3.3]{USTPDiagFReg}, where we overlooked the necessity that Frobenius be surjective on the base ring. 
\end{remark}

\begin{remark}
Diagonal $F$-splitting was first studied by Ramanathan in \cite{RamanathanDiagFSplit} (\emph{cf.} \cite[Section 1.5]{BKFsplitting}), where he defined the notion for projective varieties over an algebraically closed field. Ramanathan was interested in diagonally $F$-split varieties because of the cohomological vanishing results they enjoy. For instance, given any ample line bundles $\mathscr L$ and $\mathscr M$ on a diagonally $F$-split projective variety $X$, the map
\[
  H^0(X, \mathscr L) \otimes H^0(X, \mathscr M) \to H^0(X, \mathscr L \otimes \mathscr M)
\]
is surjective. In particular, the section ring $\bigoplus_{n \geq 0} H^0(X, \mathscr L^n)$ is generated in degree 1. We will come back to the global definition of diagonal $F$-splitting in~\Cref{sec:global}. 
\end{remark}

\section{The main theorems}

\begin{lemma}
  \label{lemma:psiInTau}
  Let $R$ be a Noetherian, $F$-finite, and strongly $F$-regular domain. Then for all $\psi\colon F^e_*R \to R$, all ideals $J \subseteq R$, and all $s > 0$, we have $\psi(F^e_* J^{\lfloor sp^e \rfloor}) \subseteq \tau(J^{s - \frac{1}{p^e}})$.
\end{lemma}

\begin{proof}
    We have an inequality
     \[
    \floor{ sp^e} \geq \ceil{sp^e} - 1 = \ceil{sp^e-1} = \left \lceil \left(s - \frac{1}{p^e}\right)p^e \right \rceil, 
  \] 
  and so $\psi(F^e_* J^{\floor{sp^e}}) \subseteq \psi(F^e_* J^{\ceil{(s - \frac{1}{p^e})p^e}})$. 
  Since $R$ is strongly $F$-regular, we know that $1$ is a test element of $R$. Then the conclusion follows from~\Cref{prop:basicFacts}, part (d).
  \end{proof}

\begin{theorem}
  Let $\kay$ be an $F$-finite field of positive characteristic, and let $R$ be a strongly $F$-regular $\kay$-algebra essentially of finite type with $\sD 2(R/\kay)$ $F$-split. Let $\mathfrak a\subseteq R$ be an ideal, and let $s, t > 0$ with $s + t \in\bZ$. Then we have
  \[
    \mathfrak a^{s + t}  \subseteq \tau( \mathfrak a^{s - \epsilon} ) \tau( \mathfrak a^{t- \epsilon})
  \]
  for all $\epsilon \in (0, \min\{s, t\} ]$. 
  \label{thm:mainthm}
\end{theorem}

\begin{proof}
  Pick $e$ such that $p^e > 1/\epsilon$. Since $\sD 2(R/\kay)$ is $F$-split, there is an $F$-splitting $\vp \colon F^e_* R\to R$ with a lift $\hat \vp: F^e_* (R\otimes_\kay R) \to R\otimes_\kay R$. Then
  \[
    \mathfrak a^{s + t}  = \mathfrak a^{s + t}\vp(F^e_* 1)  \subseteq \vp(F^e_* \mathfrak a^{(s+t)p^e}) \subseteq \vp(F^e_* \mu_\kay (\mathfrak a^{\floor{sp^e}} \otimes_\kay \mathfrak a^{\floor{tp^e}})) = \mu_\kay( \hat \vp(F^e_*( \mathfrak a^{\floor{sp^e}} \otimes_\kay \mathfrak a^{\floor{tp^e}}))).
  \]
  By~\Cref{lemma:homOfTensor}, we can write  $\vp = \sum_i \vp_{1, i} \otimes \vp_{2, i}$ for some $\vp_{1, i}, \vp_{2,i} \in \homgp_R(F^e_*R, R)$. Then, applying~\Cref{lemma:psiInTau} gives us 
  \[
    \hat \vp(F^e_* (\mathfrak a^{\floor{sp^e}} \otimes_\kay \mathfrak a^{\floor{tp^e}})) \subseteq \tau(\mathfrak a^{s - \frac{1}{p^e}}) \otimes_\kay \tau(\mathfrak a^{t - \frac{1}{p^e}}),
  \]
  and so we have
  \[
    \mathfrak a^{s + t} \subseteq  \tau(\mathfrak a^{s - \frac{1}{p^e}}) \cdot \tau(\mathfrak a^{t - \frac{1}{p^e}}) \subseteq \tau(\mathfrak a^{s - \epsilon}) \cdot \tau(\mathfrak a^{t - \epsilon}). 
\qedhere
  \]
\end{proof}

The following bit of arithmetic comes up in the proof of the next theorem. We state it on its own for the sake of clarity. 

\begin{lemma}
  For all $x\in \bR$, there exists some $\epsilon >0$ such that $x - \epsilon > \ceil{x} - 1$. \label{lemma:epsilonCeil}
\end{lemma}

\begin{proof}
  If $x$ is an integer, then $\ceil x = x$ and the statement is obvious. Otherwise, $\ceil x -1 = \floor x < x$ and the statement follows.
\end{proof}

\begin{theorem}
  \label{thm:strongerSymbolicPowers}
  Let $\kay$ be an $F$-finite field of positive characteristic, and let $R$ be a strongly $F$-regular $\kay$-algebra essentially of finite type, with $\sD 2(R/\kay)$ $F$-split.  Let $\mathfrak p \subseteq R$ be a prime ideal, and let $h$ be the height of $\fp$. Then we have 
  \begin{equation}
    \mathfrak p^{(N)} \subseteq \fp^{(\ceil{Ns} - h)} \mathfrak p^{(\ceil{N(1-s)} - h)} \label{eq:mainContainment}
  \end{equation}
  for all $N  > 2h$ and all $s \in \left( \frac hN, 1 - \frac hN \right)$. Further, we have
  \begin{equation}
    \mathfrak p^{(2hn)}  \subseteq \fp^{n} \label{eq:USTP}
  \end{equation}
  for all $n \geq 1$. 
\end{theorem}

\begin{proof}
Note that we must have $N> 2h$ for the interval $\left( \frac hN, 1 - \frac hN \right)$ to be nonempty, and that the numbers $Ns - h$ and $N(1-s) - h$ are positive. Symbolic powers and ordinary powers of maximal ideals are the same, so we may assume that $\mathfrak p$ is not maximal. Given any such $N$ and $s$, we have
  \[
    \fp^{(N)} \subseteq \tau\left( \left( \fp^{(N)} \right)^{s - \epsilon} \right)\tau\left( \left( \fp^{(N)} \right)^{1-s - \epsilon} \right),
  \]
  by~\Cref{thm:mainthm}. Note that 
  \[
    N(s - \epsilon) -h + 1 > \ceil{Ns}-h
  \]
  for $\epsilon$ small enough, by~\Cref{lemma:epsilonCeil}. As $\ceil{Ns} - h$ is an integer, this means that 
  \[
    \floor{N(s - \epsilon)} -h+1 = \floor{N(s - \epsilon) -h+1} \geq  \ceil{Ns}-h.
  \]
  It follows from ~\Cref{lemma:tauSymbolicPower} that
  \[
    \tau\left( \left( \fp^{(N)} \right)^{s - \epsilon} \right) \subseteq  \fp^{(\floor{N(s - \epsilon)} -h + 1)} \subseteq \fp^{(\ceil{Ns} - h)}.
  \]
  By symmetry, we have
  \[
    \tau\left( \left( \fp^{(N)} \right)^{1-s - \epsilon} \right) \subseteq \fp^{(\ceil{N(1-s)} - h)},
  \]
  as desired. 

  Plugging $N = 2hn + 1$ and $s = \frac{1}{2n}$ into Equation \eqref{eq:mainContainment} gives us 
  \begin{equation*}
    \mathfrak p^{(2hn+1)} \subseteq \fp \fp^{(2h(n-1)+1)},
  \end{equation*}
  so by induction on $n$ we have
  \[
    \forall n \geq 1\colon \quad \fp^{(2hn+1)} \subseteq \fp^{n} \fp^{(1)} = \fp^{n+1}.
  \]
  This shows Equation \eqref{eq:USTP} for $n \geq 2$, as $\fp^{(2h(n+1))} \subseteq \fp^{(2hn + 1)}$. The containment $\fp^{(2h)} \subseteq \fp$ is clear from the definitions.
\end{proof}

\begin{remark}
  If $n \geq 2$, then taking $N = hn + 1$ and $s = \frac{1}{n}$ in Equation \eqref{eq:mainContainment} yields $\fp^{(hn+1)} \subseteq \fp \fp^{(h(n-2)+1)}$. 
\end{remark}

\begin{remark}
  More generally, if $R$ is just strongly $F$-regular, then we have
  \[
    \mathfrak a^{s + t}  \tau( \sD 2(R/\kay)) \subseteq \tau( \mathfrak a^{s} ) \tau( \mathfrak a^{t})
  \]
  for all $s, t \geq 0$ with $s+t\in \bZ$ (\emph{cf.} \cite[Proposition 3.4]{USTPDiagFReg}). Then an argument as in~\Cref{thm:strongerSymbolicPowers} yields
  \[
    \forall n \geq 2\colon\quad \mathfrak p^{(hn)}  \tau(\sD 2(R/\kay)) \subseteq \mathfrak p \mathfrak p^{(h(n-2)+1)}; 
  \]
   the lack of epsilons ``$\epsilon$'' here allows us to remove the ``$+1$'' from the symbolic power on the left. It follows that 
  \[
    \forall n \geq 1\colon\quad \fp^{(2h(n+1))}\tau(\sD 2(R/\kay))^n \subseteq \fp^{n+1}.
  \]
\end{remark}

\section{Applications to determinantal rings} 
Our goal in this section is to explain how~\Cref{thm:mainthm,thm:strongerSymbolicPowers} apply to determinantal rings. In particular, we show the following.  

\begin{theorem}
  Let $\kay$ be an $F$-finite field and $X_{m\times n}$ an $m\times n$ matrix of indeterminates over $\kay$. Choose integers $p \leq m$ and $q \leq n$, as well as integers 
  \begin{align*}
    1 \leq u_1 < \cdots < u_p \leq m, \quad & 1 \leq r_1 < \cdots < r_p \leq m,\\
    1 \leq v_1 < \cdots < v_q \leq n, \quad & 1 \leq s_1 < \cdots < s_q \leq n. 
  \end{align*}
  Let $I$ be the ideal generated by the size $r_i$ minors of the first $u_i$ rows of $X$ as well as the size $s_j$ minors of the first $v_i$ columns of $X$, for all $i = 1, \ldots, p$ and $j = 1, \ldots, q$ \textup{(}with the convention that any size $t$ minor of a size $u \times v$ matrix is 0 if $t > \min\{u, v\}$\textup{)}.  Then  $R = \kay[X_{m\times n}]/I$ is strongly $F$-regular and $\sD 2(R/\kay)$ is $F$-split. In particular, we have $\fp^{(2hn)} \subseteq \fp^n$ for all $\fp \in \Spec R$ of height $h$ and for all $n$. 
  \label{thm:detlDiagFSplit}
\end{theorem}

Proving this theorem is mostly a matter of collecting results found in the literature. The argument goes as follows:
Generalized Schubert varieties over algebraically closed fields are known to be globally $F$-regular and diagonally $F$-split, and any open affine subscheme of a globally $F$-regular (respectively, diagonally $F$-split) scheme is strongly $F$-regular (respectively, diagonally $F$-split). Each ring described in~\Cref{thm:detlDiagFSplit} is an open affine subscheme of some Schubert subvariety of a Grassmannian, and thus these rings are strongly $F$-regular and their diagonal Cartier algebras $\sD 2(R/\kay)$ are $F$-split, at least if the base field $\kay$ is algebraically closed. It is well known that strong $F$-regularity can be checked after changing the base field; we show in~\Cref{cor:diagFSplitBaseChg} that the $F$-splitting of $\sD 2(R/\kay)$ can be checked after base changing from $\kay$ to any perfect field. Let us begin by reviewing some of this language. 

\subsection{Global definitions} \label{sec:global}
Global $F$-regularity is, as the name suggests, a global notion of strong $F$-regularity: a projective variety $X$ over an $F$-finite field is said to be \emph{globally $F$-regular} if it admits an ample line bundle $\mathscr L$ such that the section ring $\bigoplus_{n\geq 0} H^0(X, \mathscr L^n)$ is strongly $F$-regular as a ring. This definition is due to Smith in \cite{SmithGlobalFReg}. There, she shows that the section ring of \emph{any} ample line bundle on a globally $F$-regular variety is strongly $F$-regular. Further, if $X$ is a globally $F$-regular variety, then any affine subscheme of $X$ is strongly $F$-regular (as a ring), though the converse can fail. 

Next, we define what it means for a projective variety to be diagonally $F$-split. Let $Y$ be a scheme of characteristic $p$. We define the \emph{absolute Frobenius morphism} on $Y$ to be the map $F\colon Y \to Y$ which is the identity on topological spaces and for which  the map on structure sheaves is given by
\[
    F^\#(U)\colon \mathscr O_Y(U) \lra F_* \mathscr O_Y(U), \quad x \longmapsto x^p
\]
for all open $U \subseteq Y$. We say that $Y$ is $F$-split if there is a map $\vp\colon F_* \mathscr O_Y \to \mathscr O_Y$ which splits  $F^\#$. Given any closed subscheme $Z \subseteq Y$ with ideal sheaf $\mathscr I_Z$, we say $Y$ is $F$-split \emph{compatibly with $Z$} if there exists such a splitting which further satisfies $\vp(F_* \mathscr I_Z) \subseteq \mathscr I_Z$. 

If $X$ is a projective variety over a field $\kay$ of characteristic $p$, then the diagonal subscheme $\Delta_{X/\kay} \subseteq X \times_\kay X$ is closed. In the literature, such a variety $X$ is defined to be \emph{diagonally $F$-split} if $X \times_\kay X$ is $F$-split compatibly with $\Delta_{X/\kay}$. More generally, if $Y$ is any separated scheme over a base scheme $S$ of characteristic $p$, we can define $Y$ to be \emph{diagonally $F$-split over $S$} if $Y\times_S Y$ is $F$-split compatibly with the diagonal subscheme $\Delta_{Y/S} \subseteq Y \times_S Y$.  

\begin{lemma}
  Suppose $X$ is a diagonally $F$-split $S$-scheme and $U$ is an open subscheme of\, $X$. Then $U$ is diagonally $F$-split over $S$. 
  \label{lemma:openDiagFSplit}
\end{lemma}

\begin{proof}
  The diagonal $\Delta_{U/S} \subseteq U\times_S U$ is the same as $\Delta_{X/S} \cap(U \times_S U)$. 
  Since $X \times_S X$ is $F$-split compatibly with $\Delta_{X/S}$, it follows that $U \times_S U$ is $F$-split compatibly with $\Delta_{U/S}$, by \cite[Lemma 1.1.7]{BKFsplitting}.
\end{proof}

\subsection{Schubert varieties} 
Let $G$ be a connected and simply connected semisimple linear algebraic group over a field $\kay$. Let $B$ be a Borel subgroup of $G$ and $P  \supseteq B$ a parabolic subgroup of $G$. Then $B$ acts algebraically on $G/P$ on the left with finitely many orbits.  The closure of any one of these orbits is called a Schubert subvariety of $G/P$.  If $\kay$ has positive characteristic, then any  Schubert subvariety $X \subseteq G/P$  is globally $F$-regular; see \cite[Theorem 2.2]{GloballyFRegSchubert}. If $\kay$ is further algebraically closed, then any Schubert subvariety $X \subseteq G/P$ is diagonally $F$-split; see \cite[Theorem 3.5]{RamanathanDiagFSplit}. 

If we take $G$ to be  $\GL_n(\kay)$, $B$ to be the subgroup of upper-triangular matrices, and $P$ to be the subgroup of $G$ given by
\[
  P = \left\{ (g_{ij}) \in \GL_n(\kay) \mid g_{ij} = 0 \textrm{ whenever } i > d \textrm{ and } j \leq d \right\}
\]
for some fixed $d < n$, then $G/P$ is $\Gr_d(\kay^n)$,
the Grassmannian of $d$-dimensional subspaces of $\kay^n$. The Schubert subvarieties of Grassmannians are the so-called ``classical'' Schubert varieties. The rings described in~\Cref{thm:detlDiagFSplit} are open subschemes of Schubert subvarieties of $\Gr_m(\kay^{m+n})$; see \cite[Theorem 5.5]{BrunsVetter}. For the reader's convenience, we will sketch a proof of this beautiful fact in the remainder of this subsection. A more detailed treatment is found in \cite[Chapters~4 and 5]{BrunsVetter}. 

Let $\mathscr M \subseteq \bA_\kay^{m(m+n)}$ be the set of $m\times (m+n)$ matrices with full rank. Then every $m$-dimensional subspace of $\kay^{m+n}$ can be expressed as the row-span of some $M \in \mathscr M$. Note that two different matrices have the same row-span if and only if one can be obtained from the other by a sequence of elementary row operations.  

Given any list of integers $1 \leq b_1 < \cdots < b_m \leq n+m$, let $\delta_{b_1, \ldots, b_m}\colon \mathscr M \to \kay$ be the function sending a matrix $M$ to the determinant of the submatrix of $M$ obtained by taking columns $b_1, b_2, \ldots, b_m$. Note that this is a maximal minor of $M$, and there are $\binom{m+n}{m}$ possible choices of such maximal minors. Set $N = \binom{m+n}{m}$ and consider the map $\rho\colon \mathscr M \to \bP^{N-1}$ sending each matrix to a list of its maximal minors,
\[
  \rho(M) = ( \delta_{1, 2,  \ldots,  m}(M): \dots: \delta_{n+1, n+2, \ldots, n+m}(M)),
  \]
in some fixed order. Note that one of these minors must be nonzero, as $M$ has maximal rank. Further, one checks that two matrices $M, N \in \mathscr M$ have the same row-span precisely when $\rho(M) = \rho(N)$. Thus we can identify the Grassmannian $\Gr_m(\kay^{m+n})$ with the image of $\rho$.  This is known as the \emph{Pl\"ucker embedding}. 

We get a partial ordering $\preceq$ on the set of minors $\delta_{b_1,\ldots, b_m}$ by declaring 
\[
  \delta_{b_1, \ldots, b_m} \preceq \delta_{c_1, \ldots, c_m} \Longleftrightarrow b_i \leq c_i \textrm{ for } i = 1, \ldots, m.
\]
In this way we also have a partial ordering on the coordinates of $\bP^{N-1}$. Then the Schubert subvarieties of $\Gr_m\left( \kay^{m+n} \right)$ are the subvarieties cut out by certain subsets of these coordinates of $\bP^{N-1}$, namely the subsets which form a ``poset ideal cogenerated by a single element'' under this partial ordering; see \cite[Theorem 5.4]{BrunsVetter}. 

It turns out that $\Gr_m(\kay^{m+n})$ intersected with any standard open affine $D(x_i)\subseteq \bP^{N-1}_\kay$ is a copy of $mn$-dimensional affine space. For instance, the inverse image of $D(x_N) \cap \Gr_m(\kay^{m+n})$ under $\rho$ is the set of matrices in $\mathscr M$ where the determinant of the $m$ rightmost columns is nonzero. Up to row operations, any such matrix can be uniquely written in the form
\begin{equation} \label{eq:bigmatrix}
  \begin{bmatrix}
    \alpha_{11} &\cdots & \alpha_{1n} & 0 & 0 & \cdots & 1\\
    \vdots      &       & \vdots      & \vdots & \vdots &  & \vdots\\
    \alpha_{(m-1)1} &\cdots & \alpha_{(m-1)n} & 0 & 1 & \cdots & 0\\
    \alpha_{m1} &\cdots & \alpha_{mn} & 1 & 0 & \cdots & 0\\
  \end{bmatrix},
\end{equation}
and conversely any choice of $\alpha_{ij}$ gives an element of $\rho\invrs\left( D(x_N) \right)$. 

Now, the key observation is that \emph{every} minor of the $m \times n$ matrix $\left( \alpha_{ij} \right)$ equals a maximal minor of the full $m \times (n+m)$ matrix in Equation \eqref{eq:bigmatrix}. This means that for any collection of minors $\delta_1, \ldots, \delta_r \in \kay[X_{m\times n}]$, one can find a corresponding collection of coordinates $x_{i_1}, \ldots, x_{i_r}$ on $\bP^{N-1}$ so that
\[
  \Spec\left( \kay\left[ X_{m\times n} \right]/(\delta_{1}, \ldots, \delta_r) \right) = V(x_{i_1}, \ldots, x_{i_r}) \cap \Gr_m(\kay^{m+n}) \cap D(x_N). 
\]
All that is left to check is that the coordinates vanishing along the varieties in~\Cref{thm:detlDiagFSplit} always form a poset ideal cogenerated by a single element, which is done in  \cite[Theorem 5.3]{BrunsVetter}. 

\subsection{A base-change lemma for diagonal \texorpdfstring{$F$}{F}-splitting}
The above work proves~\Cref{thm:detlDiagFSplit} whenever $\kay$ is algebraically closed.  For the general case, we just need to show that the $F$-splitting of $\sD 2(R/\kay)$ can be checked after base change to an algebraically closed field. We prove a more general fact, which we hope might be of independent use. 

\begin{proposition}
  Let $A$ be an $F$-finite field, and let $R$ be an $A$ algebra essentially of finite type. Suppose $B$ is a perfect field extension of $A$, and let $I \subseteq R$ be an ideal. Set $R_B \coloneqq R \otimes_A B$ and $I_B \coloneqq I\otimes_A B = IR_B$. If $\restrictedCartAlgSpecifyRFull{R_B}{I_B}$ is $F$-split, then $\restrictedCartAlgFull$ is $F$-split. If $A$ is perfect, then the converse holds as well. 
  \label{lemma:baseChange}
\end{proposition}

\begin{proof} Suppose that $\restrictedCartAlgSpecifyRFull{R_B}{I_B}$ is $F$-split. We wish to find a surjective map $\psi \in \homgp_R(F^e_* (R/I), R/I)$ with a lift to $R$. Let $\pi$ denote the map $\pi\colon R \surj R/I$, so that we have $F^e_* \pi\colon F^e_*R \surj F^e_* (R/I)$.  Let $(F^e_* \pi)^*$ be the functor $\homgp_R(-, R/I)$ applied to $F^e_* \pi$, and let $\pi_*$ be the functor $\homgp_R(F^e_* R, -)$ applied to $\pi$. Observe that a map $\psi \in \homgp_R( F^e_* (R/I), R/I)$ has a lift to $R$ if and only if $(F^e_* \pi)^*(\psi) = \psi \circ F^e_* \pi$ is in the image of~$\pi_*$. Further, $(F^e_* \pi)^*(\psi)$ is surjective precisely when $\psi$ is.  Thus we wish to show that  $\image( (F^e_*\pi)^*) \cap \image(\pi_{*})$ contains a surjection, or  equivalently that the ``evaluation at 1'' map is surjective. To emphasize, this is the key point: $\restrictedCartAlgFull$ is $F$-split precisely when the map
    \[
       \eval_1\colon  \image( (F^e_*\pi)^*) \cap \image(\pi_{*}) \lra R/I, \quad \vp \longmapsto \vp(F^e_* 1)
    \]
    is surjective. 

    Because $R_B$ is faithfully flat over $R$, the map $\eval_1$ is surjective precisely when $\eval_1 \otimes_R R_B$ is, which is equivalent to $\left( \image( (F^e_*\pi)^*) \cap \image(\pi_{*})\right) \otimes_R R_B$ containing a surjection. This same faithful flatness yields
    \begin{align*}
    \left( \image( (F^e_*\pi)^*) \cap \image(\pi_{*})\right) \otimes_R R_B  &= \image( (F^e_*\pi)^*) \otimes_R R_B  \cap (\image \pi_*) \otimes_R R_B\\
    &= \image( (F^e_*\pi)^* \otimes_R R_B) \cap \image (\pi_* \otimes_R R_B)\\
    &= \image( ((F^e_* \pi) \otimes_A B)^*) \cap \image (\pi \otimes_A B)_*,
  \end{align*}
    where $((F^e_* \pi) \otimes_A B)^*$ is the functor $\homgp_{R_B}(-, R_B/I_B)$ applied to $(F^e_* \pi) \otimes_A B$, 
    and  $(\pi \otimes_A B)_*$ is  the functor $\homgp_{R_B}(F^e_* R \otimes_A B, -)$ applied to $\pi \otimes_A B$. 

    By assumption, there exists a surjection $\psi\colon F^e_* (R_B/I_B) \to R_B/I_B$ which lifts to $R_B$:
  \[
    \xymatrix{
      F^e_* R_B \ar[r]^{\widehat \psi} \ar[d]_{F^e_*(\pi\otimes_A B)}  &  R_B \ar[d]^{\pi\otimes_A B} \\
      F^e_* (R_B/I_B) \ar[r]^{\psi}  & R_B/I_B\rlap{.}
    } 
  \]
  Observe that we have a natural $R_B$-linear map $\widehat \alpha\colon F^e_* R\otimes_A B \to F^e_* R_B$ given by $F^e_* r \otimes b \mapsto F^e_* (r \otimes b^{p^e})$---this is just the composition of the natural maps $F^e_* R \otimes_A B \to F^e_* R \otimes_A F^e_*B$ and $F^e_* R \otimes_A F^e_*B \to F^e_*(R\otimes_A B)$ described in Section 2.  Similarly, we have a natural map $\alpha\colon F^e_* (R/I) \otimes_A B \to F^e_* (R_B/I_B)$. Because $B$ is perfect, both $\alpha$ and $\hat \alpha$ are surjective. These maps fit into a larger diagram
  \[
    \xymatrix{
      F^e_* R \otimes_A B\ar[r]^{\widehat\alpha} \ar[d]_{(F^e_* \pi) \otimes_A B} & F^e_* R_B \ar[r]^{\widehat \psi} \ar[d]_{F^e_*(\pi\otimes_A B)}  &  R_B \ar[d]^{\pi\otimes_A B} \\
      F^e_* (R/I) \otimes_A B \ar[r]^{\alpha} & 
      F^e_* (R_B/I_B) \ar[r]^{\psi}  & R_B/I_B\rlap{.}
    } 
  \]
  Thus $\psi \circ \alpha \circ ((F^e_* \pi) \otimes_A B)$ is a surjection in $\image((F^e_* \pi) \otimes_A B)^*\cap \image (\pi \otimes_A B)_*$, as desired. 

  For the converse, suppose that $A$ is perfect and $\restrictedCartAlgFull$ is $F$-split. By the above argument, we have a surjection in $\psi \in \image((F^e_* \pi) \otimes_A B)^*\cap \image (\pi \otimes_A B)_*$. Since $B$ is perfect, the Frobenius map $B \to F^e_* B$ is an isomorphism. Combining this with the assumption that $A$ is perfect, we see that the maps $\alpha$ and $\widehat \alpha$ are isomorphisms. Setting $\psi = \chi \circ (F^e_* \pi \otimes_A B)$, we see that $\chi \circ \alpha\invrs$ is an $F$-splitting in $\restrictedCartAlgSpecifyRFull{R_B}{I_B}$.
 \end{proof}

\begin{corollary}
  \label{cor:diagFSplitBaseChg}
  Let $A$ be an $F$-finite field, and let $R$ be an $A$ algebra essentially of finite type. Suppose $B$ is a perfect field extension of $A$. If $\sD 2 (R_B/B)$ is $F$-split, then $\sD 2(R/A)$ is $F$-split.
\end{corollary}

\begin{proof}
  By definition, the Cartier algebra $\restrictedCartAlgSpecifyRFull{R_B \otimes_B R_B}{\ker \mu_B}$ is $F$-split, where $\mu_B$ is the multiplication map $\mu_B \colon R_B \otimes_B R_B \to R_B$. But $R_B \otimes_B R_B = R\otimes_A R \otimes_A B$, and we can obtain $\mu_B$ by applying the functor $- \otimes_A B$ to the multiplication map $\mu_A\colon R\otimes_A R \to R$. By faithful flatness, we see that $\ker \mu_B = (\ker \mu_A) \otimes_A B$. Then the result follows from~\Cref{lemma:baseChange}. 
\end{proof}

\newcommand{\etalchar}[1]{$^{#1}$}
\providecommand{\bysame}{\leavevmode\hbox to3em{\hrulefill}\thinspace}
\providecommand{\MR}{\relax\ifhmode\unskip\space\fi MR }
\providecommand{\MRhref}[2]{%
  \href{http://www.ams.org/mathscinet-getitem?mr=#1}{#2}
}
\providecommand{\href}[2]{#2}

\end{document}